\documentclass[11pt]{amsart}

\usepackage{amssymb}
\usepackage{amsmath}
\usepackage{amsthm}
\usepackage{graphicx,psfrag}
\usepackage{enumerate}

\usepackage[dvips]{hyperref}
\hypersetup{pdfborder={}} 

\theoremstyle{plain}
\newtheorem{theorem}{Theorem}[section]
\newtheorem{lemma}[theorem]{Lemma}

\theoremstyle{definition}

\theoremstyle{remark}
\newtheorem{remark}[theorem]{Remark}
\newtheorem*{remark*}{Remark}

\numberwithin{figure}{section}

\begin{document}

\title{Small 3-manifolds with large Heegaard distance}
\author{Tao Li}
\address{Department of Mathematics \\
 Boston College \\
 Chestnut Hill, MA 02467}
\email{taoli@bc.edu}
\thanks{Partially supported by an NSF grant}

\begin{abstract}
We construct examples of closed non-Haken hyperbolic 3-manifolds with a Heegaard splitting of arbitrarily large distance. 
\end{abstract}

\maketitle

\section{Introduction}\label{Sintro} 
A Heegaard splitting of a closed orientable 3-manifold $M$ is a decomposition of $M$ into two handlebodies along an embedded surface called Heegaard surface.  
A very useful tool in studying Heegaard splittings is the curve complex. Let $F$ be a closed orientable surface of genus at least 2.  The curve complex of $F$, introduced by Harvey \cite{Har}, is the complex whose vertices are the isotopy classes of essential simple closed curves in $F$, and $k+1$ vertices determine a $k$-simplex if they are represented by pairwise disjoint curves.  We denote the curve complex of $F$ by $\mathcal{C}(F)$.  For any two vertices $x$, $y$ in $\mathcal{C}(F)$, the distance $d(x,y)$ is the minimal number of 1-simplices in a simplicial path jointing $x$ to $y$.  If $F$ bounds a handlebody $V$, then the disk complex $\mathcal{D}(V)$ is the subcomplex of $\mathcal{C}(F)$ containing only the vertices represented by boundary curves of compressing disks in $V$.  Given a Heegaard splitting $M=V\cup_F W$, the distance $d(V,W)$, introduced by Hempel \cite{He}, is the distance between $\mathcal{D}(V)$ and $\mathcal{D}(W)$ in the curve complex $\mathcal{C}(F)$.

An incompressible surface in $M$ is an embedded surface which contains no essential loop bounding a compressing disk in $M$.  Both Heegaard surface and incompressible surface play important roles in the study of 3-manifold topology, and there are some intriguing connections between them, for example, see \cite{L1, L2}.  
In \cite{H}, Hartshorn proved that if the distance $d(V,W)$ is large, then $M$ contains no small-genus incompressible surface.  A natural question is whether there is a 3-manifold $M$ that has large Heegaard distance but contains no incompressible surface at all, i.e., $M$ is non-Haken (or small).  A positive answer to this question is expected, but it is surprisingly hard to construct a concrete example.

There have been many constructions of non-Haken 3-manifolds, for example \cite{FH, HT, O}.  However, in some sense, the 3-manifolds in these constructions are relatively simple, but large Heegaard distance usually means that the 3-manifold is complicated.  On the other hand, many complicated 3-manifolds do contain incompressible surfaces, e.g. \cite{FM, LM1}.

In this paper, we construct the first examples of closed non-Haken hyperbolic 3-manifolds with a Heegaard splitting of arbitrarily large distance.  Note that by a theorem of Scharlemann and Tomova (also see \cite{L3, L4}) if the distance of a Heegaard splitting $M=V\cup_F W$ is sufficiently large, then any minimal-genus Heegaard surface is isotopic to $F$.  

\begin{theorem}\label{Tmain}
For any $g$, there are closed orientable non-Haken 3-manifolds with a genus-$g$ Heegaard splitting of arbitrarily large distance.
\end{theorem}

The motivation for this research is to explore some deeper relation between Heegaard surfaces, incompressible surfaces and geometry of the 3-manifolds.  Many interesting questions in 3-manifold topology concerning non-Haken 3-manifolds are still open and it is possible that one can use Heegaard splittings to study these questions.

Our main construction is a combination of the constructions in \cite{A} and \cite{LM}.  In \cite{A}, Agol constructed a small link which is a pure braid in $S^2\times S^1$.  We basically take Agol's construction in $S^2\times I$ and consider its double branched cover.  The double branched cover is homeomorphic to $F\times I$ and we use it as a neighborhood of our Heegaard surface $F$.  To guarantee the Heegaard distance is large, we also use the construction and criterion given by Lustig and Moriah \cite{LM} on Heegaard distance.

We organize the paper as follows.  In section~\ref{SA}, we briefly explain Agol's construction and show that this construction also works for an immersed surface which lifts to an incompressible surface in a double branched cover.  In section~\ref{SLM}, we review some results in \cite{LM}.  We finish the proof of Theorem~\ref{Tmain} in section~\ref{SP}.

\section{Agol's construction and its double branched cover}\label{SA}

We first briefly review a construction of Agol in \cite{A}.  The operation in \cite{A} is on a pure braid in $S^2\times S^1$, or equivalently, on $\Sigma\times S^1$, where $\Sigma$ is the $n$-punctured sphere.  In this paper, we only consider $\Sigma\times I$.  We use the same notations as those in \cite{A} and refer the reader to \cite{A} for more detailed discussions and proofs.

As in \cite{A}, we start with a pants decomposition of $\Sigma$ and use a special path in the pants decomposition graph $\mathcal{P}(\Sigma)^{(1)}$ (see \cite[Definition 2.1]{A}).  

Let $D_n$ be a regular $n$-sided polygon and $\gamma$ its boundary. 
As in \cite{A}, we view $\Sigma$ as a 2-sphere obtained by gluing two copies of $D_n$ along $\gamma$, with punctures at the $n$ vertices.  We cyclically label the $n$ edges of $D_n$ by $1,\dots, n$. For each pair of edges $\{i, j\}$, there is a loop $\beta_{i,j}$ in $\Sigma$ which meets $\gamma$ exactly twice at the edges $i$ and $j$.  

We start with an initial pants decomposition of $\Sigma$ given by a set of loops $P_0=\{\beta_{1,3},\beta_{1,4},\dots,\beta_{1, n-1}\}$.  Let $C_0$ be a path of pants decompositions in $\mathcal{P}(\Sigma)^{(1)}$ denoted by $C_0=P_0\to P_1\to\cdots\to P_{n-3}$, where $P_{n-3}=\{\beta_{2,4},\beta_{2,5},\dots,\beta_{2, n}\}$ and $P_k$ is obtained by replacing the first $k$ loops of $P_0$ by the first $k$ loops of $P_{n-3}$, as illustrated in Figure~\ref{Fpants} (also see \cite[Figure 5]{A}).  Moreover,we can continue the path $C_0$ in the same fashion to get a closed path $C$ as follows.  Let $P_{j(n-3)}=\{\beta_{j+1,j+3},\beta_{j+1,j+4},\dots,\beta_{j+1, j+n-1}\}$ where indices are taken (mod $n$). For each $k$, let $P_{j(n-3)+k}$ be obtained by replacing the first $k$ loops of $P_{j(n-3)}$ with the first $k$ loops of $P_{(j+1)(n-3)}$.  As $P_{n(n-3)}=P_0$.  This gives a closed path $C=P_0\to P_1\to\cdots\to P_{n(n-3)}$ in $\mathcal{P}(\Sigma)^{(1)}$.

\begin{figure}
\begin{center}
\psfrag{1}{$1$}
\psfrag{2}{$2$}
\psfrag{3}{$3$}
\psfrag{4}{$4$}
\psfrag{5}{$5$}
\psfrag{6}{$6$}
\psfrag{A}{$P_0$}
\psfrag{B}{$P_1$}
\psfrag{C}{$P_2$}
\psfrag{D}{$P_3$}
\includegraphics[width=3in]{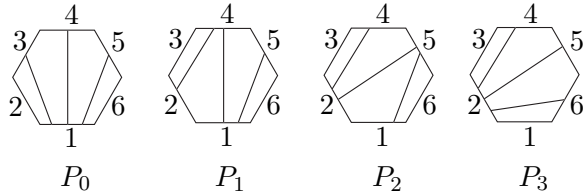}
\caption{Pants decompositions in the path $C_0$ with $n=6$}\label{Fpants}
\end{center}
\end{figure}

Each pair of pants in the pants decomposition $P_k$ is either a twice-punctured disk in $\Sigma$ bounded by $\beta_{i,j}$ with $i-j\equiv\pm 2$ (mod $n$), or a once-punctured annulus between the circles $\beta_{i,j}$ and $\beta_{i,j+1}$ or between $\beta_{i,j}$ and $\beta_{i+1,j}$.

As in \cite{A}, let $B_{i,j}$ be a horizontal loop in $\Sigma\times (0,1)$ corresponding to the circle $\beta_{i,j}$, and these horizontal loops are positioned so that $B_{i,j}$ and $B_{i-1, j+1}$ are at the same level.   Moreover, the pants in the pants decompositions $P_k$ ($k=1,\dots, n(n-3)$) can be viewed as a collection of twice-punctured disks and once-punctured annuli bounded by these curves $B_{i,j}$ and forming a 2-complex whose 1-skeleton is the union of these curves $B_{i,j}$, see \cite[Figures 7, 8, 9]{A}.  We use $X$ to denote this 2-complex.

Let $M_C$ be the manifold obtained from $\Sigma\times I$ by drilling out these loops $B_{i,j}$.  
Let $Z$ be the union of the $n$ punctures of $\Sigma$ and let $\Gamma=Z\times I$ be the braid in $S^2\times I$.  Here we view $\Sigma\cup Z= S^2$, $(\Sigma\times I)\cup\Gamma=S^2\times I$, and view $M_C\cup\Gamma$ as the manifold obtained from $S^2\times I$ by drilling out the loops $B_{i,j}$.

\vspace{10pt}
\noindent
\textbf{Double branched cover}
\smallskip

Next we consider a double branched cover of $S^2\times I$.

Let $F$ be a closed orientable surface of genus $g\ge 2$.  Let $\iota\colon  F\to F$ be a standard involution/rotation which gives rise to a standard double branched covering $\phi\colon  F\to S^2$ with $2g+2$ branch points.  We may extend $\phi$ to a double branched covering $\phi\colon F\times I\to S^2\times I$ with branch set a union of $2g+2$ vertical arcs of the form $\{x\}\times I$.  

Let $Z$ be the union of the branch points in $S^2$ and $\Gamma=Z\times I$ the branch set in $S^2\times I$.  By viewing $Z$ as punctures and letting $\Sigma=S^2-Z$, we can construct a collection of horizontal loops $B_{i,j}$'s in $\Sigma\times I$ as above.  Let $\mathcal{B}$ be the union of these loops $B_{i,j}$'s in $\Sigma\times I$ and let $\hat{\mathcal{B}}=\phi^{-1}(\mathcal{B})$ be the preimage of these loops in $F\times I$.

Let $S$ be a compact orientable incompressible surface properly embedded in $(F\times I)-\hat{\mathcal{B}}$.  The main goal in this section is to study its image $\phi(S)$ in $M_C\cup\Gamma$.  

Since each pair of pants $P$ in the 2-complex $X$ described above is isotopic to a pair of pants in a pants decomposition of $\Sigma$, after some homotopy on $\phi$ if necessary, we may assume that its preimage $\phi^{-1}(P)$ is an embedded surface with boundary in $\hat{\mathcal{B}}$ and $\phi\colon\phi^{-1}(P)\to P$ is a double branched covering.  
A pair of pants $P$ in $X$ is either a once-punctured annulus or a twice-punctured disk with punctures at $\Gamma$. 
If $P$ is a twice-punctured disk, then as shown in Figure~\ref{Fcover}(a), $\phi^{-1}(P)$ is an annulus with two punctures at $\phi^{-1}(\Gamma)$; if $P$ is a once-punctured annulus, then as shown in Figure~\ref{Fcover}(b), $\phi^{-1}(P)$ is a pair of pants with one puncture at $\phi^{-1}(\Gamma)$. 
 As $S$ is incompressible in $(F\times I)-\hat{\mathcal{B}}$, after isotopy, we may assume that no curve of $S\cap\phi^{-1}(P)$ bounds a (punctured) disk in $\phi^{-1}(P)$.  So in either case, we may isotope $S$ so that each component of $S\cap\phi^{-1}(P)$ is parallel to a boundary curve of $\phi^{-1}(P)$.  Thus after isotopy, each component of $\phi(S)\cap P$ lies in a neighborhood of a curve $B_{i,j}$ in $\partial P$.

\begin{figure}
\begin{center}
\psfrag{a}{(a)}
\psfrag{b}{(b)}
\psfrag{p}{$\phi$}
\psfrag{i}{$\iota$}
\includegraphics[width=4in]{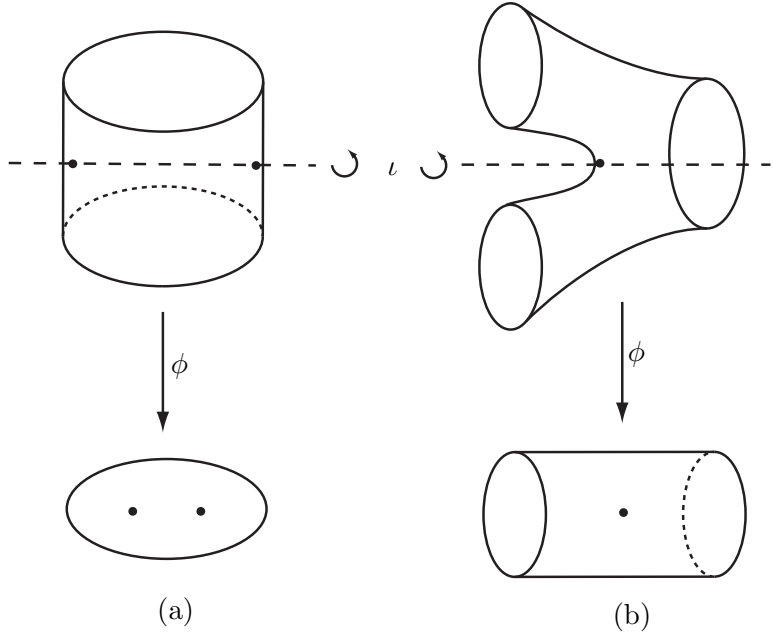}
\caption{Involution $\iota$ and double branched covering $\phi$ on pants}\label{Fcover}
\end{center}
\end{figure}

Notice that $\phi(S)$ has some nice properties. 
If there is an essential loop $c$ in $S$ such that $\phi(c)$ can be isotoped (in $M_C$) into a small neighborhood of a component of $\Gamma$, then $c$ bounds a disk in $F\times I$ possibly with a puncture at $\phi^{-1}(\Gamma)$.  Since $S$ is incompressible in $(F\times I)-\hat{\mathcal{B}}$, this means that $c$ cannot be essential in $S$, contradicting the assumption on $c$.  Thus, for any essential loop $c$ in $S$, $\phi(c)$ cannot be isotoped (in $M_C$) into a small neighborhood of a component of $\Gamma$.

In \cite{A}, Agol considers such a 2-complex formed by pants in $\Sigma\times S^1$, but our 2-complex $X$ lies in $\Sigma\times I$.  Nonetheless, after an isotopy pushing $\Sigma\times\partial I$ onto $X$, we may view $\Sigma\times\partial I$ as a union of some pants in $X$.  Note that if one glues $\Sigma\times\{0\}$ to $\Sigma\times\{1\}$, one obtains the same complex in $\Sigma\times S^1$ as the one in \cite{A}.  In particular, we may assume that the complement of $X$ in $\Sigma\times I$ consists of so-called $A$-regions in \cite{A}, see \cite[Figures 2 and 8]{A}.

Let $N(\mathcal{B})$ be a small neighborhood of $\mathcal{B}$.  For every pair of pants $P$ in the 2-complex $X$ above, since each component of $\phi(S)\cap P$ lies in a neighborhood of a curve $B_{i,j}$ in $\partial P$, we may assume that $\phi(S)\cap X\subset N(\mathcal{B})$.  Hence $\phi(S)-N(\mathcal{B})$ lies in the complement of $X$.  

In \cite{A}, Agol shows that if $\phi(S)$ is an embedded incompressible surface in $\Sigma\times I$, then each component of $\phi(S)-N(\mathcal{B})$ either is parallel to a pair of pants in $X$ or can be obtained by tubing together two pairs of pants in $X$ along some $B_{i,j}$.  This implies that $\phi(S)$ can be obtained by tubing together copies of these pants around the $B_{i,j}$'s.  The proof in \cite{A} is basically an argument of normal surfaces for the special cell-decomposition of $M_C$ determined by the pants in $X$, see \cite[Proof of Lemma 2.3]{A}.  It is shown in \cite{A} that a connected normal surface in each component of $M_C-X$ (i.e. an $A$-region in \cite{A}) is either a pair of pants parallel to a pair of pants in $X$ or obtained by tubing together two pairs of pants in $X$.  

Although the same normal surface argument should also work if $\phi(S)$ is not embedded, it is difficult to deal with immersed surfaces directly.  So we provide an alternative argument here.  Recall that the complement of $X$ in $\Sigma\times I$ consists of $A$-regions (see \cite[Figures 2 and 8]{A}).  Hence the closure of each component of $(\Sigma\times I)-X\cup N(\mathcal{B})$ can be viewed as a product $G\times J$ where $J=[a,b]$ is an interval, $\partial G\times J\subset\partial N(\mathcal{B})$, and each component of $G\times\partial J$ is obtained by tubing together two pairs of pants in $X$ using an annulus in $\partial N(\mathcal{B})$.  We use $G'\times J$ to denote $\phi^{-1}(G\times J)$.  So each component of $G'\times\partial J$ is obtained by tubing together the preimages of two pairs of pants in $X$.  Depending on the types of pants in $G\times\partial J$ (i.e., once-punctured annulus or twice-punctured disk), a component of $G'\times\partial J$ is either a one-hole torus or a 4-hole sphere or a two-hole torus with punctures at $\phi^{-1}(\Gamma)$, see Figure~\ref{Fcover} for pictures of the double branched covers of pants in $X$.  Since we are focused on the surface $S$, in the argument below, we ignore $\phi^{-1}(\Gamma)$ in $G'\times J$.

Let $G'\times J$ be as above and let $P'$ be a component of $G'\times\partial J$. 
Since $\phi(S)\cap X\subset N(\mathcal{B})$, $S\cap P'$  either is empty or consists of parallel copies of the curves which divide $P'$ into two subsurfaces that are the preimages of the two pairs of pants in $X$ used to form the component $\phi(P')$ of $G\times\partial J$.  Next we consider the intersection of $S$ with the level surfaces $G'\times\{t\}$, and the argument is similar to those in \cite{FH, HT}.

\vspace{10pt}
\noindent
\emph{Case (a)}.  $G'\times\{a\}$ is not a (punctured) two-hole torus.
\smallskip

After pushing curves on $S$ into $N(\hat{\mathcal{B}})$, we may assume that, for each component $Q$ of $S\cap (G'\times J)$, if $Q$ contains a curve parallel to a curve in $\hat{\mathcal{B}}$, then this curve is peripheral in $Q$.  We claim that there is a level surface $G'\times\{s\}$ that can be isotoped disjoint from $S$.  Suppose the claim is false. 
  Denote the two components of $G'\times\partial J$ by $P'$ and $P''$, so $S\cap P'\ne\emptyset$ and $S\cap P''\ne\emptyset$.  By the configuration of the $A$-regions (see \cite[Figures 2 and 8]{A}) and by the conclusion on $S\cap P'$ before Case (a), the curves in $S\cap P'$ are not parallel to the curves in $S\cap P''$.  We may assume that the height function $h:G'\times J\to J$ is a Morse function on $S\cap(G'\times J)$.  When the level surface $G'\times\{t\}$ moves from top to bottom, its intersection with $S$ changes to different curves only after passing a saddle tangency and the curve change can be viewed as adding a band, see \cite[Figure 8]{FH} for a picture. 
Since we are in the case that $P'$ is not a (punctured) two-hole torus, $P'$ is either a one-hole torus or a 4-hole sphere.  Similar to the arguments in \cite{FH, HT}, in the cases of one-hole torus and 4-hole sphere, adding a band to $S\cap P'$ either (1) produces a trivial curve bounding a disk in $F\times I-N(\hat{\mathcal{B}})$ or (2) creates a peripheral curve in $P'$.  This implies that $S\cap(G'\times\{t\})$ cannot change to a different type of essential curves unless, at some level $s\in J$, $S\cap(G'\times\{s\})$ consists of trivial and peripheral curves in $G'\times\{s\}$.  Recall that $S\cap P'\ne\emptyset$, $S\cap P''\ne\emptyset$ and curves in $S\cap P'$ and $S\cap P''$ are not parallel, so there must exist such a level $G'\times\{s\}$.  Moreover, by our assumption on $S\cap (G'\times J)$ above, if a curve in $S\cap(G'\times\{s\})$ is peripheral in $G'\times\{s\}$, then it is also peripheral in $S\cap (G'\times J)$.  Since $S$ is incompressible, this means that $S$ can be isotoped disjoint from this level $G'\times\{s\}$.
 
Thus, after isotopy, there is a level surface $G'\times\{s\}$ disjoint from $S$.  By a theorem of Waldhausen \cite[Proposition 3.1 and Corollary 3.2]{W}, this implies that each component of $S\cap (G'\times J)$ is parallel to a subsurface of $G'\times\partial J$.  Therefore, similar to Agol's conclusion in the case that $\phi(S)$ is an embedded surface, each component of $\phi(S)-X$ is isotopic to either a branched cover of a pair of pants in $X$ or a surface obtained by tubing together branched covers of two pairs of pants in $X$ around some $B_{i,j}$.  

\vspace{10pt}
\noindent
\emph{Case (b)}.  $G'\times\{a\}$ is a (punctured) two-hole torus.
\smallskip

This case is similar, though not as straightforward as Case (a).  Let $P'$ and $P''$ be the two components of $G'\times\partial J$.  As shown in Figure~\ref{Fcurve}(a), there are two simple closed curves $\alpha'$ and $\beta'$ in $P'$ that decompose $P'$ into two pairs of pants, each of which is the preimage of a once-punctured annulus in $X$.  Similarly, as shown in Figure~\ref{Fcurve}(b), there are also such a pair of curves $\alpha''$ and $\beta''$ in $P''$.  Moreover, by the configuration of an $A$-region, $\pi(\alpha')\cap\pi(\alpha'')$ and $\pi(\alpha')\cap\pi(\beta'')$ are single points, where $\pi\colon G'\times J\to G'$ is the projection.  

\begin{figure}
\begin{center}
\psfrag{a}{(a)}
\psfrag{b}{(b)}
\psfrag{c}{(c)}
\psfrag{x}{$\alpha'$}
\psfrag{y}{$\beta'$}
\psfrag{t}{$\beta''$}
\psfrag{p}{$\alpha''$}
\includegraphics[width=4.5in]{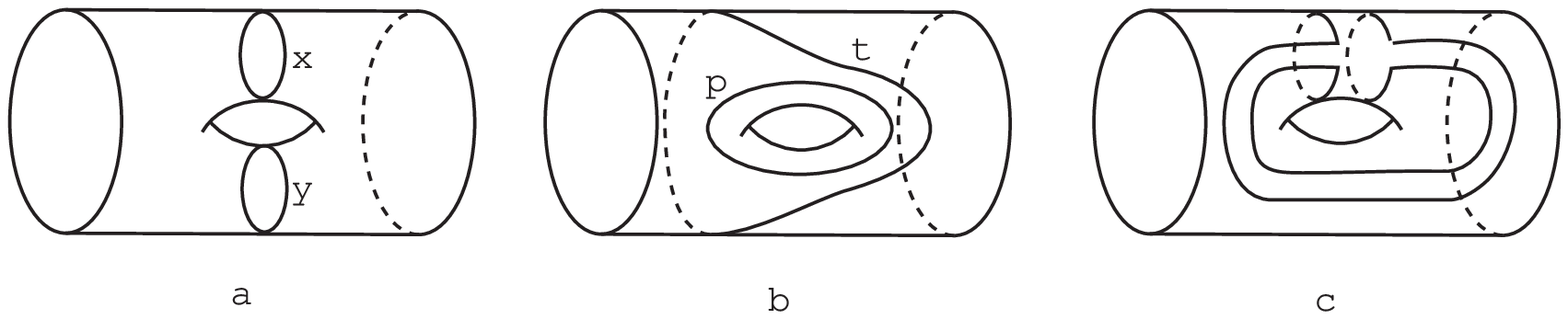}
\caption{}\label{Fcurve}
\end{center}
\end{figure}

Similar to Case (a), after pushing curves in $S$ into $N(\hat{\mathcal{B}})$, we may assume that $S\cap(G'\times J)$ does not contain any non-peripheral curve that is parallel to a curve in $\partial G'\times\{t\}$. 
Suppose $S\cap(G'\times J)$ has a component $Q$ that is not parallel to a subsurface of $G'\times\partial J$.  
So $Q\cap P'$ and $Q\cap P''$ consist of curves parallel to the curves in $\alpha'\cup\beta'$ and $\alpha''\cup\beta''$ respectively.   If $Q\cap P'$ contains curves parallel to both $\alpha'$ and $\beta'$ in $P'$, then similar to the argument in Case (a) and those in \cite{FH, HT}, adding a band to curves in $Q\cap P'$ creates a trivial or peripheral curve in $P'$, and since $Q$ is incompressible, this implies that $Q$ is parallel to a subsurface of $P'$, contradicting our assumption on $Q$.  Thus we may assume that $Q\cap P'$ consists of curves all parallel to $\alpha'$, and similarly, $Q\cap P''$ consists of curves all parallel to $\alpha''$.  Let  $N(\pi(\alpha')\cup\pi(\alpha''))$ be the closure of a small neighborhood of  $\pi(\alpha')\cup\pi(\alpha'')$ in $G'$, where $\pi\colon G'\times J\to G'$ is the projection.  As $\pi(\alpha')\cap\pi(\alpha'')$ is a single point,  $N(\pi(\alpha')\cup\pi(\alpha''))$ is a one-hole torus.  Let $C=\partial  N(\pi(\alpha')\cup\pi(\alpha''))$ be its boundary curve, see Figure~\ref{Fcurve}(c) for a picture of $C$.  So $C$ divides $G'$ into a one-hole torus, which we denote by $T$, and a pair of pants, denoted by $R$.  By our assumption on $Q\cap P'$ and $Q\cap P''$ above, we may assume that $Q\cap (C\times J)$ consists of essential simple closed curves in the vertical annulus $C\times J$.  Hence, after isotopy, $Q\cap (R\times J)$ consists of pairs of pants of the form $R\times\{x\}$.  Now we apply the arguments in Case (a) and \cite{FH} on $T\times J$, and we can conclude that, after pushing curves in $Q$ onto $C\times J$ (via annulus compression), each component of $Q\cap(T\times J)$ becomes a surface parallel to a subsurface of $T\times\partial J$.  Since we have assumed that $Q$ is not parallel to a subsurface of $G'\times\partial J$, by our assumption on $Q\cap(R\times J)$, this implies that $Q$ must be obtained by tubing together two $\partial$-parallel pairs of pants in $T\times J$ along $C\times J$.  More precisely, $Q=Q'\cup Q''\subset T\times J$ where (1) $Q'$ is a pair of pants with two boundary circles in $P'$ parallel to $\alpha'$, (2) $Q''$ a pair of pants with two boundary circles in $P''$ parallel to $\alpha''$, and (3) their common boundary circle $C_Q=\partial Q'\cap\partial Q''$ is parallel to $C\times\{t\}$.  Next we rule out this configuration of $Q$.

Since $Q$ is incompressible, $Q$ must be $\pi_1$-injective in $T\times J$.  Note that $\pi_1(T\times J)$ is a free group generated two elements $a$ and $b$, where $a$ and $b$ are represented by $\alpha'$ and $\alpha''$ respectively.  Let $c$ be the element represented by $C_Q$.  So we may view $\pi_1(Q')$ as a free group generated by $a$ and $c$, and view $\pi_1(Q'')$ as a free group generated by $b$ and $c$.  As $Q$ is obtained by gluing $Q'$ and $Q''$ together along $C_Q$, $\pi_1(Q)$ is a rank-3 free group generated by $a$, $b$, and $c$.  However, since $c=aba^{-1}b^{-1}$ in $\pi_1(T\times J)$, the map $i_*\colon\pi_1(Q)\to \pi_1(T\times J)$ (which is induced by the inclusion) cannot be injective, contradicting that $Q$ is incompressible.

Therefore, each component of $S\cap(G'\times J)$ must be parallel to a subsurface of $G'\times\partial J$. So in both Case (a) and Case (b), each component of $\phi(S)-X$ is isotopic to either a branched cover of a pair of pants in $X$ or a surface obtained by tubing together branched covers of two pairs of pants in $X$ around some $B_{i,j}$.   This means that the whole surface $\phi(S)$ can be obtained by tubing together branched covers of these pants around the curves $B_{i,j}$ in $\mathcal{B}$.
 
Let $N$ be the manifold obtained from $F\times I$ by performing a $\frac{1}{n_i}$-Dehn surgery on each curve $l_i$ in $\hat{\mathcal{B}}$.  As $S$ is disjoint from $\hat{\mathcal{B}}$, we may view $S\subset N$.  Next we suppose $S$ is incompressible in $N$ and consider the isotopy of $S$ in $N$.  Similar to \cite{A}, by $\frac{1}{n_i}$-Dehn surgery on these curves, $S$ can be isotoped across curves in $\hat{\mathcal{B}}$ without affecting the isotopy class of $S$ in $N$.  So, as in \cite{A}, $\phi(S)$ can be isotoped across the curves $B_{i,j}$ in $\mathcal{B}$ without affecting the isotopy class of $S$ in $N$.  Thus, to build $\phi(S)$ using pants in $X$, it makes no difference which ways to tube together (the branched covers of) these pants around the $B_{i,j}$'s.

In \cite{A}, it is also shown that, if $\phi(S)$ is an embedded surface, no component of $\phi(S)$ can be obtained by tubing together a collection of once-punctured annuli, in other words, at least one piece must be a twice-punctured disk.  This is because of the particular choice of the path $C$ in $\mathcal{P}(\Sigma)^{(1)}$ of the pants decompositions.  If $\phi(S)$ is obtained by tubing together only once-punctured annuli, then after isotoping $\phi(S)$ across curves in $\mathcal{B}$, $\phi(S)$ has two once-punctured annuli tubed together as in \cite[Figure 12]{A}, which gives a loop in $\phi(S)$ bounding a once-punctured disk, contradicting the property of $\phi(S)$ that we mentioned at the beginning.  The same argument also works for the case that $\phi(S)$ is not embedded.  If $\phi(S)$ is obtained by tubing together only branched covers of once-punctured annuli, then after isotoping $\phi(S)$ across curves in $\mathcal{B}$, there are two once-punctured annuli as in \cite[Figure~12]{A}, such that the branched covers of the two once-punctured annuli are tubed together.  The loop bounding the once-punctured disk in \cite[Figure~12]{A} gives rise to an essential loop in $\phi(S)$ winding around the meridian of a component of $\Gamma$, contradicting the properties of $\phi(S)$ that we mentioned at the beginning.  We would like to remark that the argument in Case (b) above also implies this fact, because the double branched cover of the region in \cite[Figure 12]{A} is of the form $T\times J$ where $T$ is a one-hole torus and the argument in Case (b) says that the surface obtained by such tubing must be compressible in $T\times J$.

Therefore, when we tube together the branched covers of these pants together to obtain $\phi(S)$, at least one pair of these pants must be a twice-punctured disk.  For Euler characteristic reason, this means that each component of $\phi(S)$ must be a branched cover of either a punctured disk or a punctured sphere with branch points at $\Gamma$.  In particular, after isotopy, there is a level surface $\Sigma\times\{t\}$ disjoint from $\phi(S)$.

\section{Fat train tracks and distance of Heegaard splittings}\label{SLM}

In this section, we review some results of Lustig and Moriah.  In \cite{LM},  Lustig and Moriah defined a fat train track which is similar to the fibered neighborhood of a traditional train track.  We first briefly describe the fat train track used in \cite{LM} and refer the reader to \cite[sections 2 and 3]{LM} for more details.  

A \emph{train track} $\tau$ in \cite{LM} is defined as a closed subsurface of $F$ with a singular $I$-fibration as follows: the interior of $\tau$ is fibered by open intervals and the fibration extends to a fibration of $\tau$ by properly embedded closed intervals, except for finitely many singular (or cusp) points on $\partial\tau$, where precisely two $I$-fibers meet.  Such fibers are called \emph{singular fibers}.  Note that both endpoints of an $I$-fiber may be singular/cusp points.  Two singular $I$-fibers are adjacent if they share a singular point as a common endpoint.  A maximal connected union of singular $I$-fibers is called an \emph{exceptional fiber}.  In \cite{LM}, an exceptional fiber is allowed to be a circle.  In fact, a train track $\tau$ is called a \emph{fat train track} if all of its exceptional fibers are cyclic.  The collection of all the exceptional fibers of a fat train track $\tau$ is denoted by $\mathcal{E}_\tau$.  A train track $\tau$ is \emph{maximal} if its complement is a collection of triangles.

An arc or a closed curve $\alpha$ in $\tau$ is \emph{carried} by $\tau$ if it is transverse to the $I$-fibers of $\tau$, and $\alpha$ \emph{covers} $\tau$ if (1) $\alpha$ is carried by $\tau$ and (2) $\alpha$ meets every $I$-fiber of $\tau$.  Note that this is the same as saying that $\alpha$ is fully carried by $\tau$ in some other papers. 

Similar to the usual notion of train tracks, one can split a fat train track.  Note that the train track after splitting may not be fat.  The type of splittings that we are interested in are splittings without changing the complement of $\tau$ (this is called splitting without collision in \cite{LM}).  Let $\alpha$ be an arc carried by $\tau$ and with precisely one endpoint at a singular/cusp point.  Then one can split (or unzip) $\tau$ along $\alpha$ and get a new train track $\tau'$.  Since only one endpoint of $\alpha$ is at the cusp, $F-\tau$ is isotopic to $F-\tau'$ in $F$.  In particular, if $\tau$ is maximal, so is $\tau'$.  We say $\tau'$ is \emph{derived from} $\tau$ if $\alpha$ covers $\tau$.  Note that we may view $\tau\supset\tau'$.

Given a closed surface $F$, a complete decomposing system $\mathcal{D}$ of $F$ is a collection of simple closed curves in $F$ that decompose $F$ into  pairs of pants.  A fat train track $\tau$ in $F$ is called \emph{complete} if the following conditions are satisfied.
\begin{enumerate}
  \item The collection $\mathcal{E}_\tau$ of exceptional fibers of $\tau$ is a complete decomposing system of $F$.
  \item Each pair of pants $P$ complementary to the system $\mathcal{E}_\tau$ contains two triangles as complementary components of $\tau$ in $P$.
  \item The train track $\tau$ only carries seams, but no waves, with respect to $\mathcal{E}_\tau$ (recall that a seam in a pair of pants is a properly embedded arc with endpoints in different boundary circles and a wave is a properly embedded arc with both endpoints in the same boundary circle.)
\end{enumerate}

Let $H$ be a handlebody. In \cite{LM}, the authors consider complete decomposing systems of $\partial H$ which bound disk systems in $H$, and denote the set of such isotopy classes by $\mathcal{CDS}(H)$. Using the work of Masur and Minsky in \cite{MM}, Lustig and Moriah proved the following.

\begin{theorem}[Theorem 4.7 in \cite{LM}]\label{TLM}
Let $M$ be an orientable 3-manifold with a Heegaard splitting $M=V\cup_F W$.  Consider complete decomposing systems $\mathcal{D}\in\mathcal{CDS}(V)$ and $\mathcal{E}\in\mathcal{CDS}(W)$ which do not have waves with respect to each other.  Let $\tau\subset F$ be a complete fat train track with exceptional fibers $\mathcal{E}_\tau=\mathcal{E}$, and assume that $\mathcal{D}$ is carried by $\tau_n$ for some $n$-tower of derived train tracks $\tau=\tau_0\supset\tau_1\supset\cdots\supset \tau_n$ with $n\ge 2$.  Then the distance of the given Heegaard splitting satisfies $d(V,W)\ge n$.
\end{theorem}

Note that, by \cite[Lemma 2.5]{LM}, if one curve in $\mathcal{D}$ covers $\tau_n$ in Theorem~\ref{TLM}, then every curve in $\mathcal{D}$ is carried by $\tau_n$, and by \cite[Lemma 2.9]{LM}, if a curve is carried by $\tau_n$, then it covers $\tau_{n-1}$. 
We would also like to quote some results in \cite{LM} which are useful for our purposes.  

\begin{lemma}[Lemma 3.9 in \cite{LM}]\label{LML}
Let $\tau$ be a complete fat train track on a surface $F$, and let $\tau'$ be a train track derived from $\tau$. Let $\beta$ be an arc with endpoints on $\mathcal{E}_\tau$ which covers $\tau'$. Let $D$ be an essential simple closed curve which is tight with respect to $\mathcal{E}_\tau$ and contains $\beta$ as a subarc. Then $D$ can be carried by $\tau'$, and in fact covers $\tau'$.
\end{lemma}

By applying Lemma~\ref{LML} to Dehn twists, Lustig and Moriah \cite{LM} proved the following.  Lemma~\ref{LMP} is a weaker version of \cite[Proposition 3.12]{LM}.  

\begin{lemma}[Proposition 3.12 of \cite{LM}]\label{LMP}
Let $\mathcal{E}$ be a complete decomposing system on $F$ and let $\tau'$ be a maximal train track that is derived from some complete fat train track $\tau$ with exceptional fibers $\mathcal{E}_\tau = \mathcal{E}$.  Let $k$ be an essential simple closed curve that covers $\tau'$.  Let $\mathcal{D}$  be either an essential simple closed curve that non-trivially intersects $k$ or a complete decomposing system on $F$.
Let $\delta_k^m\colon F\to F$ be the $m$-fold Dehn twist along $k$. 
 Then there exists an integer $m_0$ such that for every $m \in\mathbb{Z}\setminus\{m_0, m_0+1, m_0+2, m_0 +3\}$, $\delta_k^m(\mathcal{D})$ covers $\tau'$.
\end{lemma}
\begin{remark}
Lemma~\ref{LMP} is basically half of  \cite[Proposition 3.12]{LM}.  In this lemma, since both $k$ and $\delta_k^m(\mathcal{D})$ cover $\tau'$, $k$ and $\delta_k^m(\mathcal{D})$ have no wave with respect to $\mathcal{E}$.  In the case that $\mathcal{D}$ is a complete decomposing system, since we do not need the result that $\mathcal{E}$ has no wave with respect to $\delta_k^m(\mathcal{D})$, we do not need the hypothesis in \cite[Proposition 3.12]{LM} that $k$ has no wave with respect to $\mathcal{D}$, see part (I) of the proof of \cite[Proposition 3.12]{LM}.
\end{remark}

\section{Proof of the main theorem}\label{SP}

The first step of our construction is similar to that in \cite[section 5.3]{LM}. 
We start with a standard genus $g$ Heegaard splitting $S^3=V\cup_F W$ of $S^3$.  Let $\mathcal{D}\in\mathcal{CDS}(V)$ and $\mathcal{E}\in\mathcal{CDS}(W)$ be complete decomposing systems.  In \cite[section 5.3]{LM}, it is shown that $F$ contains a minimal-maximal lamination $\mathcal{L}$ (which means that the complement of $\mathcal{L}$ consists of ideal triangles and each leaf is dense in $\mathcal{L}$) such that the leaves of $\mathcal{L}$ are tight with respect to both $\mathcal{D}$ and $\mathcal{E}$, and $\mathcal{L}$ has no wave with respect to both $\mathcal{D}$ and $\mathcal{E}$.  
As in  \cite[section 5.3]{LM}, $\mathcal{L}$ is carried by maximal fat train tracks $\tau_\mathcal{D}$ and $\tau_\mathcal{E}$ with $\mathcal{D}$ and $\mathcal{E}$ respectively as exceptional fibers.

We may split $\tau_\mathcal{E}$ along $\mathcal{L}$ and get an $n$-tower of derived train tracks $\tau_\mathcal{E}=\tau_0\supset\tau_1\supset\cdots\supset \tau_n$ ($n\ge 2$) with $\tau_n$ carrying $\mathcal{L}$.  Let $k$ be a simple closed curve close to $\mathcal{L}$ in the projective measured lamination space, and in particular, we require that $k$ covers $\tau_n$.  Let $\delta_k^m\colon  F\to F$ be the $m$-fold Dehn twist along $k$.

Now we consider the collection of loops $\hat{\mathcal{B}}$ in $F\times I$ constructed in section~\ref{SA}.  Let $\pi\colon F\times I\to F$ be the projection and we consider curves $\pi(\hat{\mathcal{B}})$ in $F$.  We may choose $k$ to be a curve intersecting every curve in $\pi(\hat{\mathcal{B}})$.  By Lemma~\ref{LMP}, for each curve $l$ in $\pi(\hat{\mathcal{B}})$, $\delta_k^m(l)$ covers $\tau_n$ for all but finitely many integers $m$.  Thus, we can find an integer $m$ such that $\delta_k^m(l)$ covers $\tau_n$ for every curve $l$ in $\pi(\hat{\mathcal{B}})$.

Next we apply the automorphism $\delta_k^m$ of $F$ and extend it to an automorphism $\delta_k^m\colon F\times I\to F\times I$.  Let $\hat{\mathcal{B}}'=\delta_k^m(\hat{\mathcal{B}})$ be a new collection of loops in $F\times I$.  So, after projected to $F$, each curve in $\hat{\mathcal{B}}'$ covers $\tau_n$.   

We may view $F\times I$ as a neighborhood of the Heegaard surface $F$ in $S^3$.  Let $l_1', l_2',\dots, l_p'$ be the sequence of loops in $\hat{\mathcal{B}}'$.  Let $M$ be the manifold obtained from $S^3$ by performing $\frac{1}{n_{i}}$-Dehn surgery on each $l_i'$ in $\hat{\mathcal{B}}'$.  Since each surgery slope is of the form $\frac{1}{m}$, a level surface $F\times\{t\}$ remains a Heegaard surface of $M$.  Our task is to show that there are Dehn surgery slopes $\frac{1}{n_i}$ such that (1) $M$ is non-Haken and (2) the Heegaard splitting of $M$ (given by the Heegaard surface $F$) has distance at least $n$.

Consider the level surface $F\times\{t\}$.  When $t$ changes from $0$ to $1$, the level surface $F\times\{t\}$ passes the sequence of horizontal loops $l_1', l_2',\dots, l_p'$ in $\hat{\mathcal{B}}'$.  
The $\frac{1}{n_{i}}$-Dehn surgery on $l_i'$ is the same as the $n_{i}$-fold Dehn twist in $F$ along the curve $l_i=\pi(l_i')$ where $\pi\colon F\times I\to F$ is the projection. 
 Thus we may view the Heegaard splitting of $M$ as obtained from $S^3=V\cup_F W$ by performing a sequence of $n_{i}$-fold Dehn twists along the curves $l_i$ in $\pi(\hat{\mathcal{B}}')$.  We use $f\colon F\to F$ to denote the composition of this sequence of Dehn twists.  In other words, $f=\delta_1^{n_1}\circ\delta_2^{n_2}\circ\cdots\circ\delta_p^{n_p}$, where each $\delta_i^{n_i}$ is the $n_{i}$-fold Dehn twist along the corresponding curve $l_i=\pi(l_{i}')$.  By our construction of $M$, we may choose $f$ so that $f(\mathcal{D})$ and $\mathcal{E}$ are curve systems in the $F$ bounding disk systems in the two handlebodies of the Heegaard splitting of $M$ given by the Heegaard surface $F$.  The main task is to determine the value of each integer $n_i$ ($i=1,\dots,p$).

Let $N$ be the manifold obtained from $S^3-N(\hat{\mathcal{B}}')$ by performing $\frac{1}{n_i}$-Dehn fillings on the $p-2$ boundary tori of $S^3-N(\hat{\mathcal{B}}')$ corresponding to the curves $l_2',\dots, l_{p-1}'$ in $\hat{\mathcal{B}}'$.  So $\partial N$ consists of two tori $T_1$ and $T_p$ corresponding to $l_1'$ and $l_p'$ respectively.  We may view $S^3-N(\hat{\mathcal{B}}')$ as a submanifold of $N$.  By a theorem of Hatcher \cite{Ha}, we may choose the Dehn filling slopes $\frac{1}{n_i}$ ($i=2,\dots, p-1$) not to be the boundary slopes of an orientable incompressible surface in $S^3-N(\hat{\mathcal{B}}')$ disjoint from $T_1\cup T_p$, in other words, any closed orientable incompressible surface in $N$ lies in $S^3-N(\hat{\mathcal{B}}')$ after isotopy.  Next we consider the Dehn filling slopes at $T_1$ and $T_p$.

We use $N_p(m)$ to denote the manifold obtained from $N$ by the $\frac{1}{m}$-Dehn filling on $T_p$.  Similar to the argument above, by Hatcher's theorem \cite{Ha}, we can find $5$ integers $m_1,\dots, m_5$ (in fact there are infinitely many such integers), such that for each $j=1,\dots, 5$, any closed orientable incompressible surface in $N_p(m_j)$ lies in $S^3-N(\hat{\mathcal{B}}')$ after isotopy.  So each $N_p(m_j)$ is a 3-manifold with $\partial N_p(m_j)=T_1$.  Next we consider Dehn filling on $N_p(m_j)$ at $T_1$.

Let $g_j=\delta_2^{n_2}\circ\cdots\circ\delta_{p-1}^{n_{p-1}}\circ\delta_p^{m_j}$ ($j=1,\dots, 5$).  
Recall that every curve $l_i$ in $\pi(\hat{\mathcal{B}}')$ covers the train track $\tau_n$. 
By Lemma~\ref{LMP}, for each $j=1,\dots, 5$, $\delta_1^m(g_j(\mathcal{D}))$ covers $\tau_n$ for all but finitely many integers $m$.  Moreover, by Hatcher's theorem \cite{Ha}, the number of boundary slopes of orientable incompressible surfaces in each $N_p(m_j)$ is finite.  Thus, we can find an integer $n_1$ such that, for every $j=1,\dots, 5$,
\begin{enumerate}
  \item  $\delta_1^{n_1}(g_j(\mathcal{D}))$ covers $\tau_n$, and  
  \item  $\frac{1}{n_1}$ is not a boundary slope of a properly embedded orientable incompressible surface in $N_p(m_j)$.
\end{enumerate}  

Let $f_j=\delta_1^{n_1}\circ g_j=\delta_1^{n_1}\circ\delta_2^{n_2}\circ\cdots\circ\delta_{p-1}^{n_{p-1}}\circ\delta_p^{m_j}$ ($j=1,\dots, 5$).  
Recall that in our construction, the lamination $\mathcal{L}$ has no wave with respect to both $\mathcal{D}$ and $\mathcal{E}$.  Hence a curve carried by $\tau_n$ contains no wave with respect to $\mathcal{E}$.  Since $f_j(\mathcal{D})$ covers $\tau_n$, $f_j(\mathcal{D})$ contains no wave with respect to $\mathcal{E}$ for each $j=1,\dots, 5$.

Now we reverse the roles of  $\mathcal{D}$ and $\mathcal{E}$.  We consider $f_j^{-1}=\delta_p^{-m_j}\circ\delta_{p-1}^{-n_{p-1}}\circ\cdots\circ\delta_1^{-n_1}=\delta_p^{-m_j}\circ h$, where $h=\delta_{p-1}^{-n_{p-1}}\circ\cdots\circ\delta_1^{-n_1}$.   Symmetrically, we consider the maximal fat train track $\tau_{\mathcal{D}}$ with $\mathcal{D}$ as its exceptional fibers.  Recall that the splitting of $\tau_{\mathcal{E}}$ is along the lamination $\mathcal{L}$ which is carried by both $\tau_{\mathcal{D}}$ and $\tau_{\mathcal{E}}$, and the curve $k$ at the beginning is assumed to be close to $\mathcal{L}$ in the projective measured lamination space.  So we may assume that $k$ also covers a maximal train track $\tau'$ derived from $\tau_{\mathcal{D}}$.  Moreover, we may choose the Dehn twist $\delta_k^m$ at the beginning of this section so that the curve $l_p$ in $\pi(\hat{\mathcal{B}}')$ also covers $\tau'$. 
By Lemma~\ref{LMP}, except for $4$ possible integers $m$, $\delta_p^{m}(h(\mathcal{E}))$ covers $\tau'$.  We have 5 integers $m_1,\dots, m_5$ to choose from, so there is an integer $m_j$ such that $\delta_p^{-m_j}(h(\mathcal{E}))$ covers $\tau'$.  Let $n_p$ be this integer $m_j$ and set $f=f_j=\delta_1^{n_1}\circ\delta_2^{n_2}\circ\cdots\circ\delta_p^{n_p}$.  Thus $f^{-1}(\mathcal{E})=\delta_p^{-n_p}(h(\mathcal{E}))$ covers $\tau'$.  As $\mathcal{L}$ contains no wave with respect to both $\mathcal{D}$ and $\mathcal{E}$, each curve carried by $\tau'$ contains no wave with respect to $\mathcal{D}$.  Hence $f^{-1}(\mathcal{E})$ contains no wave with respect to $\mathcal{D}$.  After a composition with $f$, we see that $\mathcal{E}$ has no wave with respect to $f(\mathcal{D})$. 

As $f=f_j$ for some $j\in\{1,\dots, 5\}$, $f(\mathcal{D})$ covers $\tau_n$ and contains no wave with respect to $\mathcal{E}$.  Hence $f(\mathcal{D})$ and $\mathcal{E}$ have no wave with respect to each other.

We have found all our integers $n_i$ ($i=1,\dots, p$). 
Since $f(\mathcal{D})$ covers $\tau_n$ and since $f(\mathcal{D})$ and $\mathcal{E}$ have no wave with respect to each other, by Theorem~\ref{TLM}, the distance of the Heegaard splitting of $M$ is at least $n$.  Moreover, by the properties of $n_i$ and $m_j$ above, if $M$ contains a closed orientable incompressible surface $S$, then $S$ can be isotoped disjoint from the surgery curves and we may view $S\subset  S^3-\hat{\mathcal{B}}'$.  Now Theorem~\ref{Tmain} follows from the following lemma.

\begin{lemma}\label{LHaken}
$M$ is non-Haken.
\end{lemma}
\begin{proof}
Suppose $M$ contains a closed orientable incompressible surface $S$.  By our choices of $n_i$ ($i=1,\dots,p$), we may assume that $S$ is an incompressible surface in $S^3-\hat{\mathcal{B}}'$. 

Recall that $\hat{\mathcal{B}}'=\delta_k^m(\hat{\mathcal{B}})$ for some Dehn twist $\delta_k^m$ and 
we view $F\times I$ as a submanifold of $S^3$ containing $\hat{\mathcal{B}}'$.  
  Since we changed the loops from $\hat{\mathcal{B}}$ to $\hat{\mathcal{B}}'$, we also need to change the involution $\iota$ in section~\ref{SA} to a new involution $\iota'=\delta_k^m\circ\iota\circ\delta_k^{-m}$.  So $\hat{\mathcal{B}}'$ is invariant under $\iota'$.  Let $\phi'\colon F\to S^2$ be the double branched covering determined by $\iota'$ and we extend it to a double branched covering $\phi'\colon F\times I\to S^2\times I$ as in section~\ref{SA}.   

Let $S'=S\cap(F\times I)$. Since $S$ is incompressible in $S^3-\hat{\mathcal{B}}'$, we may assume that $S'$ is incompressible in $(F\times I)-\hat{\mathcal{B}}'$.  By section~\ref{SA}, after isotoping $S$ across loops in $\hat{\mathcal{B}}'$ (which does not affect the isotopy class of $S$ in $M$), we may assume that each component of $\phi'(S')$ is a branched cover of a disk or sphere.  This implies that, after isotopy, there is a level sphere $S^2\times\{t\}$ in $S^2\times I$ disjoint from $\phi'(S')$.  Therefore, after isotopy, there is a level surface $F\times\{t\}$ in $F\times I$ disjoint from $S$.  Since each $F\times\{t\}$ is a Heegaard surface of $M$, this means that $S$ lies in one of the two handlebodies in the Heegaard splitting, contradicting that $S$ is incompressible.
\end{proof}

\end{document}